\documentclass[12pt,leqno,fleqn]{amsart}  
\usepackage{amsmath,amstext,amsthm,amssymb,amsxtra}

\usepackage{txfonts} 
\usepackage[T1]{fontenc}
\usepackage{lmodern}

 \usepackage{euler}   

\usepackage{mathtools}
\mathtoolsset{showonlyrefs,showmanualtags}

\usepackage{hyperref} 
\hypersetup{
    colorlinks=true,       
    linkcolor=blue,          
    citecolor=magenta,        
    filecolor=magenta,      
    urlcolor=cyan           
}

\usepackage[shortalphabetic,msc-links]{amsrefs}

\allowdisplaybreaks
\swapnumbers

\theoremstyle{plain} 
\newtheorem{lemma}[equation]{Lemma} 
\newtheorem{proposition}[equation]{Proposition} 
\newtheorem{theorem}[equation]{Theorem} 
 
\newtheorem{conjecture}[equation]{Conjecture}

\theoremstyle{definition}
\newtheorem{definition}[equation]{Definition} 

\theoremstyle{remark}
\newtheorem{remark}[equation]{Remark}

\newtheorem*{ack}{Acknowledgment}

\numberwithin{equation}{section}

\title {An $ A_p$--$A_ \infty $ inequality for the Hilbert Transform}
\author{Michael T. Lacey}   

\address{ School of Mathematics, Georgia Institute of Technology, Atlanta GA 30332, USA}
\email {lacey@math.gatech.edu}
\thanks{Research supported in part by grant NSF-DMS 0968499.}

\begin{document}
\begin{abstract}
We prove in particular that for the Hilbert transform,  for $ 1<p< \infty $ and a weight $ w\in A_p$, that we have the inequality 
\begin{equation*}
\lVert H \rVert_{ L ^{p} (w) \to L ^{p} (w)} 
\lesssim \lVert w\rVert_{A _p } ^{1/p } 
\max \bigl\{ 
\|  w \|_{ A _{\infty}}^{1/p'}, \lVert w ^{- p'+1} \rVert_{A _{\infty } }  ^{1/p}  \bigr\} 
\end{equation*}
The case of $ p=2$ is an instance of a recent result of Hyt\"onen-Perez, and 
as a corollary we obtain the well-known bound of S. Petermichl of $ \lVert w\rVert_{A_p} ^{ \max \{1, (p-1) ^{-1} \}}$. 
This supports a conjectural inequality valid for all Calder\'on-Zygmund operators $ T$, and $ p\neq 2$. 
\end{abstract}
\maketitle 

\section{Introduction: Main Theorem} 

We are interested in estimates for the norms of Calder\'on-Zygmund operators on weighted $ L ^{p}$-spaces, 
a question that has attracted significant interest recently;  definitive estimates 
of this type have been obtained in  \cites{1007.4330,1103.5229}, among others.   Our particular motivation here is the 
paper of Lerner \cite{MR2721744}, and   Hyt\"onen-Perez \cite{1103.5562} that focus on a quantification of 
good estimates of the norm of an operator in terms of the $ A_p$ and $ A _{\infty }$ characteristics of a weight.

\begin{definition}\label{d.Ap} Let  $ w$ be a weight on $ \mathbb R ^{d}$  with density also written as $ w$.  Assume $ w>0$ a.\thinspace e., 
and $ 1<p< \infty $. We define  $ \sigma= w ^{1-p'} = w ^{\frac 1 {1-p}}$, which is defined a.\thinspace e.\thinspace, and set 
\begin{equation} \label{e.Ap}
\lVert w\rVert_{A_p} := \sup _{Q} \frac {w (Q)} {\lvert  Q\rvert } \Bigl[ \frac {\sigma (Q)} {\lvert  Q\rvert } \Bigr] ^{p-1} \,. 
\end{equation}
 For the  endpoint  $ p= \infty $,  we set 
\begin{align}\label{e.AzI}
\lVert w\rVert_{A_ \infty }& :=  \sup _{Q} w (Q) ^{-1} \int _{Q} M (w Q) \; dx \,. 
\end{align}
\end{definition}
	
Note that the definition of \eqref{e.AzI} is different from more familiar definition of an $ A _{\infty }$ norm \cite{MR727244}. 
It originates in \cite{MR883661}, and the article of \cite{1103.5562} makes a convincing case for it's central 
role in the subject. It is known that $ \lVert w\rVert_{A _{\infty }} \le c_d \lVert w\rVert_{A_p}$ for $ 1< p < \infty $. 
The article  \cite{1103.5562}  proves a very sharp estimate on the $ L ^2 (w)$-norm of Calder\'on-Zygmund operators 
in terms of the $ A_2$ and $ A _{\infty }$ characteristics of the weight.  \emph{The main purpose of this article is to extend this 
result to the case of $ p\neq2$ for a few canonical Calder\'on-Zygmund operators.}  This result, stated just below, suggests a 
clear conjecture, which we return to in the concluding section of this paper.

Let us say that $ T$  is a classical Calder\'on-Zygmund operator if it is of the form $ T f (x) = \textup{p.v.} \int f (y) K (x-y)\; dy$ where 
(1) $ K (y)= y ^{-1} $, namely the Hilbert transform; (2) $ K (z) = z ^{-2} $, for complex $ z$, namely the Beurling operator; 
(3) powers of the Beurling operator; 
(4) $ K (y) = y \lvert y\rvert ^{-d-1}$, in dimension $ d\ge 2$, namely the Riesz transform; and lastly (4) $ K (y)$ is any odd, one-dimensional  
$ C^2$ Calder\'on-Zygmund kernel.  
By the latter, we mean that $ \lvert \partial ^{\epsilon } K (y)\rvert \lesssim \lvert  y\rvert ^{-1- \epsilon }  $ for $ \epsilon =0,1,2$. 
These operators are distinguished in that they are known to be in the convex hull 
of a Haar shift of \emph{bounded} complexity.  (See the next section for a definition.)

We define the maximal truncations of $ T$ by 
\begin{equation*}
T f (x) := \sup _{\epsilon < \delta } 
 \Biggl\lvert  \int _{\epsilon < \lvert  x-y\rvert < \delta  } f (y) K (x-y)\; dy  \Biggr\rvert\,. 
\end{equation*}
This result is new for $ p\neq 2$ with or without truncations; and in the case of $ p=2$, it is new for the maximal truncations. 

\begin{theorem}\label{t.classical} 
Let  $  T $ be a classical singular integral in the sense just defined, and $ 1<p< \infty $, and $ w\in A_p$.   It then holds that 
\begin{align}
\lVert  T_{\natural } f  \rVert _{ L ^{p} (w)} &\le C_{T,p} \|  w \|_{ A _{p}}^{1/p} 
\max \bigl\{ 
\|  w \|_{ A _{\infty}}^{1/p'}, \lVert w ^{- p'+1} \rVert_{A _{\infty } }  ^{1/p}  \bigr\} 
\lVert  f \rVert _{L ^{p} (w)}.
\end{align} 
\end{theorem}

Since we have $ \lVert w\rVert_{A _{\infty }} \lesssim \lVert w\rVert_{A_p}$, 
the result above contains the sharp estimate $ \lVert  T_{\natural }   \rVert _{  L ^{p} (w)} \lesssim \lVert w\rVert_{A_p} ^{\max \{1, 1/(p-1) ^{-1} \}}$.  This last estimate holds in complete generality, and is a central estimate of \cite{1103.5229}.  
There is a weak-$L^p$ norm analog \eqref{e.zz}  of our main inequality   that is known in complete generality \cite{1103.5229}*{Section 12}.

A key step in the proof is the analogous result for dyadic models of the operators in question, see \S\ref{s.haar}.  
This is a step that has been critical 
in many  contributions, beginning with the breakthrough result of \cite{MR2354322}, and has taken an even more central role with the 
\emph{random} Haar shifts in \cite{1007.4330}.  An important component of this definition is a notion of \emph{complexity.}  

We apply to the Haar shift operators the remarkable \emph{Lerner median inequality} \cite{MR2721744}, see \S\ref{s.lerner}.  It is well-known 
that this inequality yields estimates that are exponential in complexity.  This means that the main theorem of this section only applies 
to operators which are in the convex hull of shifts of \emph{bounded complexity.} Fortunately, this is known to include 
the few canonical examples indicated above.  For the Hilbert transform see \cite{MR1756958}; Beurling \cite{MR1992955};  
powers of the Beurling \cite{MR2281449}; 
Riesz transforms \cite{MR1964822}; and for $ C ^2 $ odd one-dimensional kernels,  see \cite{MR2680056}.  It is an open problem to extend the main result of this last paper to higher dimensions.  

The Lerner median inequality has been applied in the setting of weighted inequalities, see \cite{MR2628851}, but we follow a more 
refined path here.  After application of a the Lerner median inequality, we get a class of dyadic positive operators with particular 
structure.   There is a very precise understanding of the norms of these operators, see \cite{0911.3437}, and we recall these results in \S\ref{s.positive}, analyzing the particular operators of interest in that section.   
This analysis completes the proof.

\begin{ack}
The genesis for the ideas herein begain from my lectures at the Spring School in Analysis, Paesky, Czech Republic.  I thank the organizers for their efforts which lead to a successful School.  
\end{ack}

\section{The Haar Shift Result} \label{s.haar}

We begin our dyadic analysis. By a \emph{grid} we mean a collection $ \mathcal D$ of cubes in $ \mathbb R ^{d}$ with  $ Q \cap Q' \in \{\emptyset , Q, Q'\}$ for all $ Q, Q' \in \mathcal D$.   
The cubes can be taken to be a product of half-open half-closed intervals.
We will say that $ \mathcal D$ is a  \emph{dyadic grid} if each cube $ Q \in \mathcal D$ these two properties hold. 
(1)  $ Q$  is the union of $ 2 ^{d}$-subcubes of  equal volume (the children of $ Q$), 
and (2) the set of cubes $ \{ Q' \in \mathcal D \;:\; \lvert  Q'\rvert= \lvert  Q\rvert  \}$ partition $ \mathbb R ^{d}$.  
We will work with dyadic grids below.

We give a (standard) definition of \emph{Haar shifts}.  
By a  \emph{Haar function} we will mean a function $ h _{I}$, supported on $ I$, constant on its children, and orthogonal to $ \chi _{I}$ (and no assumption on normalizations).   
And, by a  \emph{generalized Haar function}  as a function $ h_I$ which is a linear combination of $ \chi _I$, and $ \{\chi _{I' } \;:\; I'\in \textup{Child} (I)\}$. 
Such a function supported on $ I$ but need not be orthogonal to constants.  In the definition, and throughout the paper, 
$ \ell (Q) = \lvert  Q\rvert ^{1/d} $ is the side length of $ Q$. 

\begin{definition}\label{d.haar} For integers $(m,n) \in \mathbb N		 ^2 $, we say that  linear operator $ \mathbb S $ is a \emph{ (generalized) Haar shift operator of parameters $ (m,n)$} if 
\begin{equation}\label{e.mn}
\mathbb S  f (x) = \sum_{Q \in \mathcal D}\; \sideset {} { ^ {(m,n)}} \sum_{\substack{Q',R'\in \mathcal D\\ Q',R'\subset Q }} 
\frac { \langle f, h ^{Q'} _{R'} \rangle} {\lvert  Q\rvert } h ^{R'} _{Q'} 
\end{equation}
where (1) in the second sum, the superscript $ ^{(m,n)}$ on the sum means that in addition we require $ \ell (Q') = 2 ^{-m} \ell (Q)$ and $ \ell (R')= 2 ^{-n} \ell (Q)$, and (2) the function $  h ^{Q'} _{R'}$ is a (generalized) Haar function on $ R'$, and $  h ^{R'} _{Q'}$ is one on $ Q'$, with the joint normalization that 
\begin{equation} \label{e.normal}
\lVert  h ^{Q'} _{R'}\rVert_{\infty } \lVert  h ^{R'} _{Q'}\rVert_{\infty } \le 1 \,. 
\end{equation}
In particular, this means that we have the representation 
\begin{equation} \label{e.NOR}
\mathbb S  f (x) = \sum_{Q\in \mathcal D} \lvert  Q\rvert ^{-1} \int _{Q} f (y)  s _{Q} (x,y) \; dy  
\end{equation}
where $ s_Q (x,y)$ is supported on $ Q \times Q$, with $ L ^{\infty }$ norm at most one.  
We say that the \emph{complexity} of $ \mathbb S $ is $ \kappa = \max (m,n)$.  
\end{definition}

Particular examples include Haar multipliers,  and the Haar operators central to the result 
of \cite{MR1756958}, in which the Hilbert transform is obtained as a convex combination of a Haar shift of complexity 1. 
Notice that a Haar shift using only Haar functions is necessarily bounded on $ L ^{2}$, with norm independent of the 
complexity type.  
Dyadic paraproducts are the single example in which one should use generalized Haar functions.  The complexity type in this 
case could be taken to be of type $ (0,0)$, and is explictly 
\begin{equation*}
\sum_{Q} a_Q \mathbb E _{Q} f \cdot h_Q 
\end{equation*}
where $ \lvert  a_Q\rvert\le \sqrt {\lvert  Q\rvert } $, and $ h_Q$ is a Haar function.  
In this case, we further assume that $ \mathbb S $ is an 
$ L ^2 $-bounded operator.  We comment that the analysis of Haar shifts has been central to the papers of \cite{MR1756958,MR2657437,1010.0755,MR2628851} 
among several other recent publications.  Their central role in this paper is expected.  

Define the maximal truncations of the Haar shift operator by 
\begin{equation}\label{e.Snat}
\mathbb S _{\natural} f (x) := \sup _{\epsilon >0} \Biggl\lvert  
\sum_{ Q \;:\; \ell (Q) \ge \epsilon }  \lvert  I\rvert ^{-1} \int _{I} f (y)  s _{I} (x,y) \; dy 
\Biggr\rvert
\end{equation}

Our analysis extends to the two-weight setting, and we shift to it now.  Given a pair of weights $ w, \sigma $, we set 
the two weight $ A_p$ characteristic to be 
\begin{equation}\label{e.2Ap}
\llbracket w, \sigma \rrbracket _{A_p} 
:= 
\sup _{Q}\bigl[ \frac {w (Q)} {\lvert  Q\rvert }  \bigr] ^{1/p} 
\bigl[  \frac {\sigma  (Q)} {\lvert  Q\rvert }  \bigr] ^{1/p'}  \,. 
\end{equation}
To connect this to the classical $ A_p$ setting, we would take $ \sigma = w ^{1-p'}$, where we would then have 
$
\llbracket w, \sigma \rrbracket _{A_p} = \lVert w\rVert_{A_p} ^{1/p} 
$. 
Also, in the norm inequalities in the remainder of the paper, we will have a certain asymmetry between $ w$ and $ \sigma $, one  designed so that the inequalities behave well with respect to duality. To dualize, interchange the roles of $ w$ and $ \sigma $, and exchange $ p$ for its dual index.  
This  very useful fact appears without comment below. 

As we are discussing maximal truncations, the maximal function itself will appear, given by 
\begin{equation*}
M f (x) = \sup _{t>0} (2t) ^{-d} \int _{[-t,t] ^{d}} \lvert  f (x+y)\rvert \; dy 
\end{equation*}
The $ A_p$-$A _{\infty }$ estimate for  $ M$ have been obtained in \cite{1103.5562}; these considerations are less complicated than 
those for singular integrals.  (One only needs Sawyer's two weight characterization \cite{MR676801} for the maximal function.) 

\begin{theorem}\label{t.max} For the Maximal function $ M$, index $ 1<p< \infty $, and a pair of weights $ w, \sigma $ we have 
\begin{equation}\label{e.max}
\lVert M (f \sigma )\rVert_{ L ^{p} (w)} \le 
 C_{p}   \llbracket w, \sigma \rrbracket _{A_p} \| \sigma \|_{A_ \infty } ^{1/p}  \lVert  f \rVert _{L ^{p} (\sigma)}. 
\end{equation}
\end{theorem}

Observe that the lower bound of $  \llbracket w, \sigma \rrbracket _{A_p} $ for the weak-type norm is entirely elementary, 
thus the estimate above clearly show that contribution of the $ A _{\infty }$ norm to the strong-type norm.  

This estimate is equal to or smaller  than those we are seeking to prove.  
Our main technical Theorem is then stated in the two weight language. The corresponding result for the classical operators is of course true. 

\begin{theorem}\label{t.haarShift}
For $  \mathbb S $ an $ L ^{2} (\mathbb R ^{d})$ bounded  Haar shift operator of complexity $ \kappa $, index   $1<p<\infty$,  and a pair of weights $ w, \sigma $, it holds that 
\begin{align}
\lVert   \mathbb S _{\natural} (\sigma f)  \rVert _{ L ^{p } (w)} &\le C_{\mathbb S,p}   \llbracket w, \sigma \rrbracket _{A_p} 
\bigl\{ 
\|  w \|_{ A _{\infty}}^{1/p'}+\| \sigma \|_{A_ \infty } ^{1/p}\bigr\}  \lVert  f \rVert _{L ^{p} (\sigma)}. 
\end{align}
\end{theorem}

Our estimate on the operator norm will be exponential in the complexity parameter $ \kappa $, so that we will not attempt to 
quantify it.  It is an open question if one can  improve this estimate to \emph{polynomial in complexity}, a question we 
return to in the concluding section of this paper.   In particular, a positive answer  would immediately imply 
that our main Theorem holds for \emph{all} continuous Calder\'on-Zygmund operators.  

\subsection{Deducing the Main Theorem}

If $ T$ is a classical singular integral operator, we have the following representation for it.  
There is a measure space  $ (\Omega, \mathcal A) $, and (not necessarily positive) measure $ \mu $ on $ (\Omega, \mathcal A) $, 
with $0< \lvert  \mu \rvert (\Omega ) < \infty  $, 
so that for almost all $ \omega \in \Omega $, we have a Haar shift operator $ \mathbb S  ^{\omega } $ of complexity at most one, 
so that for $ f, g $ smooth compactly supported functions with $ \textup{dist} (\textup{supp} (f), \textup{supp} (g)) >0 $, 
\begin{equation} \label{e.rep}
\langle T f , g \rangle = \langle K \ast f, g  \rangle = \int _{\Omega } \langle \mathbb S ^{\omega } f, g \rangle \; d \mu  \,. 
\end{equation}
We remark that if one is interested in the Hilbert transform \cite{MR1756958} or smooth one-dimensional odd Calder\'on-Zygmund kernels \cite{MR2680056}, 
then the measure $ \mu $ can be taken positive.  And indeed, for the Hilbert transform there is a strong form of the equality 
above \cite{MR2464252}.  For the Beurling operator, however, $ \mu $ will be complex valued \cite{MR1992955}. 
(And much of the difficulty in that paper is showing that the measure $ \mu $ is non-trivial!) 
Thus, our main result, without truncations, follows immediately from the Haar shift version. 

Concerning truncations, setting 
\begin{equation*}
\mathbb S  ^{\omega }  f = \sum_{Q \in \mathcal D} \lvert  Q\rvert ^{-1} \int _{I } s ^{\omega } _{Q} (x,y) f (y) \; dy   
\end{equation*}
let $ \mathbb S ^{\omega } _{\epsilon }$ be the corresponding operator with sum restricted to those $ q$ with $ \ell (Q) > \epsilon  $.  Apply \eqref{e.rep} with $ f, g$ so that $ \textup{dist} (\textup{supp} (f), \textup{supp} (g)) \ge \epsilon  $. 
We see that the those $ Q$ with $ \ell (Q) < c \epsilon $ make no contribution in \eqref{e.rep}. From this, it follows that 
we will have 
\begin{equation*}
 \langle T _{\epsilon } \ast f, g  \rangle = \int _{\Omega } \langle \mathbb S ^{w} _{\epsilon } f, g \rangle  \; d \mu 
\end{equation*}
where $ T _{\epsilon } $ is an operator with kernel $ K _{\epsilon } (x,y)$ so that for $ \lvert  x-y\rvert \ge C \epsilon   $ we have 
$ K _{\epsilon } (x,y)$ equals $  K (x-y)$, the un-truncated kernel. Moreover, by the normalizing condition \eqref{e.normal}, it follows that 
\begin{equation*}
\lvert K _{\epsilon } (x-y)\rvert \lesssim \epsilon ^{-d}\,,   \qquad \lvert  x-y\rvert \lesssim \epsilon \,.  
 \end{equation*}
 That is, for $ \lvert  x-y\rvert < C \epsilon  $, we have a kernel which at worst performs an average of $ f$.  
From this, we see that we have 
\begin{equation*}
T _{\natural} f \lesssim M f + \int _{\Omega } \mathbb S _{\natural} ^{\omega } f \; d \lvert  \mu \rvert \,.  
\end{equation*}
In view of the fact that the maximal function obeys better bounds, see Theorem~\ref{t.max}, 
our main theorem for the maximal truncation of classical singular integrals follows from that for Haar shifts.

\section{Dyadic Positive Operators} \label{s.positive}

We recall elements of the main results from \cite{0911.3437}, which codifies and extends the arguments of 
\cite{MR719674,MR930072}.  The latter papers of Sawyer characterized the strong-type two-weight inequalities for the 
fractional integrals, setting out important elements of the two-weighted theory.  

 Let $ \boldsymbol \tau  = \{\tau_Q \;:\; Q\in \mathcal D\}$ be non-negative constants, and define  linear operators by 
 \begin{align}
T_{\boldsymbol \tau } f \coloneqq \sum _{Q\in \mathcal D} \tau_Q \cdot \mathbb E _{Q}f \cdot \mathbf
1_{Q} \,,
\end{align}
Here, we are defining the operator $\mathbb S_\alpha$ and a  `localization' of $\mathbb S_\alpha$
corresponding to a cube $R$.

Below, we consider the $ L ^{p} (\sigma ) $ to $ L ^{p} (w )$ mapping properties of $ \operatorname T
_{\boldsymbol \tau}$, where $ 1< p < \infty $.    First, we have the weak-type inequalities.

\begin{theorem}\label{t.dyadicWeakSawyer} Let $\boldsymbol \tau  $ be non-negative constants, and 
$ w, \sigma $ weights. Let $ 1<p < \infty $. 
We have the equivalence  below. 
\begin{align}\label{e.WWeak1}
\lVert  T_{\boldsymbol \tau} (\sigma \cdot )  \rVert_{ L ^{p} (\sigma ) \mapsto L ^{p, \infty } (w )} 
&\simeq 
\mathbf T_{p'}   (w, \sigma )
\coloneqq 
\sup _{R\in \mathcal D}  w (R) ^{-1/p'} \Bigl\lVert 
 \sum _{\substack{Q\in \mathcal D\\ Q\subset R}
} \tau_Q \cdot \mathbb E _{Q}f \cdot \mathbf
1_{Q}
\Bigr\rVert_{ L ^{p'} (\sigma )}  \,.
\end{align}
\end{theorem}

There is a corresponding, harder,  strong-type characterization.  

\begin{theorem}\label{t.dyadicStrong}  Under the same assumptions as Theorem~\ref{t.dyadicWeakSawyer} we have 
the equivalences of norms below.  
\begin{align}\label{e.strong1}
\lVert  T_{\boldsymbol \tau} (\sigma \cdot )  \rVert_{ L ^{p} (\sigma ) \mapsto L ^{p } (w )} 
&\simeq 
\mathbf T_{p'}(w, \sigma )
+
\mathbf T_{p}(\sigma ,w )
\,.
\end{align}
\end{theorem}

Notice that the strong type norm is controlled by the larger of two weak-type norms.

\subsection{A Particular Class of Operators}

We consider a  class of operators,  motivated by our upcoming application of Lerner's median inequality.  

\begin{definition}\label{d.Lgood} We say that a collection $ \mathcal L$ of dyadic cubes is type $ L$ if this condition holds. 
We have for a constant $ \Lambda = \Lambda _{\mathcal L}>0$  so that 
\begin{equation}\label{e.Lambda}
\sup _{Q\in \mathcal L} \mathbb E _{Q} \operatorname {exp}\Bigl( \Lambda ^{-1} \sum_{Q' \in \mathcal L \;:\; Q'\subset Q} \mathbf 1_{Q} \Bigr)  
\le 1 \,. 
\end{equation}
For such a collection $ \mathcal L$, define $ T_{\mathcal L} f := \sum_{Q\in \mathcal L} \mathbf 1_{Q} \cdot \mathbb E _{Q} f $. 
\end{definition}

A simple sufficent condition for \eqref{e.Lambda} is e.\thinspace g. that $ \bigl\lvert  
\bigcup \{Q' \in \mathcal L \;:\; Q'\subsetneq Q\}
\bigr\rvert < \tfrac 12 \lvert  Q\rvert $.  The advantage of the condition above is that it is the minimal condition needed to 
complete our proof, and it conveniently quantifies our estimate.  

The rationale for this definition will be come clear after the discussion in \S\ref{s.lerner}. But note that the collection $ \mathcal L$ is 
`thin' in that the $ \mathcal L$-children of any cube in the collection must be `thin.'  This notion of thinness depends upon the 
constant $ \Lambda  _{\mathcal L}$; we will have to take this constant to be exponentially large in the complexity of the Haar shift we consider. 

\begin{proposition}\label{t.forLerner}  For a collection of dyadic cubes that is of type $ L$,  $ 1<p< \infty $, 
and a pair of weights $ w , \sigma $, we have 
\begin{align}\label{e.typeLweak}
\lVert T_{\mathcal L} ( \cdot \sigma )\rVert_{ L ^{p} (\sigma ) \mapsto L ^{p, \infty } (w)}  &\lesssim 
\Lambda 
\llbracket w, \sigma \rrbracket _{A_p} 
\lVert w \rVert_{A _{\infty }} ^{1/p'}\,,
\\ \label{e.typeLstrong}
\lVert T_{\mathcal L} ( \cdot \sigma )\rVert_{ L ^{p} ( \sigma ) \mapsto L ^{p } (w)}  &\lesssim  \Lambda 
 \llbracket w, \sigma \rrbracket _{A_p} 
\max \bigl\{     \lVert \sigma \rVert_{A _{\infty }} ^{1/p} \,,\,  
 \lVert w \rVert_{A _{\infty }} ^{1/p'}
\bigr\} \,. 
\end{align}
\end{proposition}

We take up the proof of this Proposition.  
Of the two estimates \eqref{e.typeLweak} and \eqref{e.typeLstrong},  it suffices to prove the weak-type result \eqref{e.typeLweak}. 
Indeed, it is the content of \eqref{e.strong1}, that the strong-type norm of $ T_{\mathcal L}$ is characterized by the maximum of the two weak-type norms, from $ L ^{p} ( \sigma ) $ to weak-$L ^{p} (w)$ and from $ L ^{p'} (w)$ to weak-$L ^{p'} (\sigma )$. In so doing, we should keep track of the   role of the measures and index $ p$.   Thus, 
\begin{align*}
\lVert T_{\mathcal L}  ( \cdot \sigma )\rVert_{ L ^{p} (\sigma ) \mapsto L ^{p } (w)} 
& \lesssim 
\lVert T_{\mathcal L} ( \cdot \sigma )\rVert_{ L ^{p} (\sigma ) \mapsto L ^{p,\infty  } (w)}
+
\lVert T_{\mathcal L} ( \cdot w  )\rVert_{ L ^{p'} (w ) \mapsto L ^{p',\infty  } (\sigma )}
\\
& \lesssim  \Lambda \bigl\{ 
\llbracket w, \sigma \rrbracket _{A_p}  \lVert \sigma \rVert_{A _{\infty }} ^{1/p} + 
\llbracket \sigma ,w \rrbracket _{A_{p'}}  \lVert w \rVert_{A _{\infty }} ^{1/p'}\bigr\}\,. 
\end{align*}
Then, \eqref{e.typeLstrong} follows as the two-weight $ A_p$ terms above are equal. 

In particular, it suffices to show that for any cube $ Q_0$, we have 
\begin{equation}\label{e.2show}
\Bigl\lVert  \sum_{ \substack{Q\in \mathcal L\\ Q \subset Q_0 }}  \mathbb E _{Q} f w  \cdot \mathbf 1_{Q} 
\Bigr\rVert_{ L ^{p'} (\sigma )} \lesssim  \Lambda 
\llbracket \sigma ,w \rrbracket _{A_{p'}}  \lVert w \rVert_{A _{\infty }} ^{1/p'}  w (Q_0) ^{1/p'}\,. 
\end{equation}

\medskip 

The steps below are the argument pioneered in \cite{MR2657437}, and have been used in \cite{1007.4330,1010.0755,1103.5229,1103.5562}.  
The details have not been presented before in the positive case, where they are much simpler. 
We make the definition of the stopping cubes.  

\begin{definition}\label{d.1stop} Let $ \mathcal D$ be a grid, $ w $ a weight.  Given a cube $ Q \in \mathcal D$, 
we set the \emph{stopping children of $ Q$, written $ \mathcal T (Q)$,} to be the maximal dyadic cubes $ Q'\subset Q$
for which $ w (Q')/\lvert  Q'\rvert> 4 w (Q)/\lvert  Q\rvert  $.    A basic property of this collection is that 
\begin{equation}\label{e.14}
\sum_{Q'\in \mathcal T(Q)} \lvert  Q'\rvert< \tfrac 1 4 \lvert  Q\rvert\,.   
\end{equation}

We set the \emph{stopping cubes of $ Q_0$} 
to be the collection $ \mathcal S = \bigcup _{j\ge 0} \mathcal S_j (Q)$, where we inductively define $ S_0 (Q_0) := \{Q_0\}$, 
and $ S _{j+1} (Q_0) = \bigcup _{Q'\in \mathcal S _{j} (Q_0)} \mathcal T (Q)$.  Thus, these are the maximal dyadic cubes, so that 
passing from parent to child in $ \mathcal S$, the average value of $ w $ is increasing by at least factor $ 4$.  
\end{definition}

We are free to assume that $ Q\subset Q_0$ for all $ Q\in \mathcal L$, where $ Q_0$ is as in \eqref{e.2show}.
Let us fix a non-negative  integer $ a$ with $ 2 ^{a} \le 2\llbracket w, \sigma \rrbracket _{A_p}  ^{p}$, and integer $ b \ge 0$, 
take $ \mathcal L _{a,b} (S)$ be those $ Q\in \mathcal L$ so that  (a) $ S \in \mathcal S (Q_0)$ is the smallest stopping cube 
which contains $ Q$; (b)  $ 2 ^{a-1} \le \frac {w (Q)} {\lvert  Q\rvert } \bigl[ \frac {\sigma (Q)} {\lvert  Q\rvert } \bigr] ^{p-1} < 2 ^{a}$; 
and (c) $ 2 ^{-b+1}  \mathbb E _{S} w\le \mathbb E _{Q} w \le 2 ^{-b+2} \mathbb E _{S} w$.  
Notice that $ a$ holds the $ A_p$ ratio fixed, and by definition of the stopping cubes, the sets $ \mathcal L _{a,b} (S)$ exhaust 
$ \mathcal L$ as the three quantities $ S \in \mathcal S (Q_0)$,  and integers $ a, b$ vary.

Setting $ \mathcal L _{a,b} = \bigcup _{S\in \mathcal S (Q_0)} \mathcal L _{a,b} (S)$, it holds that 
\begin{equation}\label{e.aTL}
\int _{Q_0} \Bigl\lvert  \sum_{\substack{Q\in \mathcal L_{a,b} \\ Q\subset Q_0 }}  \mathbf 1_{Q} \mathbb E _{Q} w \Bigr\rvert ^{p'} 
\; \sigma (dx) \lesssim   2 ^{- p' b} \lVert w\rVert_{A_ \infty }  2 ^{a (p'-1) }  w (Q_0)\,. 
\end{equation}
This is summed over $ 0\le a \le \lceil \log_2 \llbracket w, \sigma \rrbracket _{A_p} ^{p}\rceil +1$  and non-negative $ b$ to prove  \eqref{e.2show}.

The essence of the proof  of our claim \eqref{e.aTL} is this distributional estimate. 

\begin{lemma}\label{l.Ldistrib}  With the notations above we have these two distributional inequalities, universal 
 over  (1)  integers $ a,b,t$; and  (2) $ S\in \mathcal S (Q_0)$ and (3) measure $ \nu $ equal to Lebesgue measure or 
 $\sigma $
\begin{align}
 \label{e.Ldist2}
\nu  \Bigl( \Bigl\{ x\in S \;:\; \sum_{\substack{Q\in \mathcal L _{a,b} (S)  }} \mathbb E _{Q} w  \mathbf 1_{Q}
>  K  \Lambda  2 ^{-b} t \mathbb E _{S} w
\Bigr\} \Bigr) &\lesssim \operatorname e ^{-t} \nu  (S) \,. 
\end{align} 
Here, $ K$ is a  constant. 
\end{lemma}

\begin{proof} Fix data of the Lemma, and note that it suffices to prove that for a maximal $ Q_0 \in \mathcal L _{a,b} (S)$, that 
the estimate above holds on the cube $Q_0$.  Then,   \eqref{e.Ldist2}, for $ \nu $ being Lebesgue measure, obviously reduces to 
\begin{equation*}
\Bigl\lvert \Bigl\{ x\in Q_0 \;:\; \sum_{Q\in \mathcal L _{a,b} (S)} \mathbf 1_{Q}
> K \Lambda  t 
\Bigr\} \Bigr\rvert \lesssim \operatorname e ^{-t} \lvert  Q_0\rvert 
\end{equation*}
which is an immediate consequence of $ \mathcal L$ being of type $ L$, see \eqref{e.Lambda}.  
So, we consider the case of $ \nu $ being $ \sigma $.  But note that with $ t>1$ fixed, the event 
\begin{equation*}   
 \Bigl\{ x\in Q_0\;:\; \sum_{Q\in \mathcal L _{a,b} (S)} \mathbb E _{Q} w  \mathbf 1_{Q}
> K \Lambda  t 2 ^{-b} \mathbb E _{S} w
\Bigr\}
\end{equation*}
is a union of disjoint cubes in collection $ \mathcal E_t \subset \mathcal L _{a,b}$. Now, for each $ Q\in \mathcal E_t$, we have held the $ A_p$ characteristic essentially constant, namely about $ 2 ^{a}$.  And $ \mathbb E _{Q} w$ is a held to be approximately $ 2 ^{-b} \mathbb E _{S} w$.  Thus, $ \sigma (Q)$ is, up to a fixed multiple, a constant times $ \lvert Q\rvert$.  Call this multiple $ \rho $, which is a function of the stopping cube $ S$, but the latter is held fixed. Hence, 
\begin{align*}
\sigma \Bigl\{ x\in S \;:\; \sum_{Q\in \mathcal L _{a,b} (S)} \mathbb E _{Q} w  \mathbf 1_{Q}
> K t 2 ^{-b} \mathbb E _{S} w
\Bigr\}
&= 
\sum_{Q\in \mathcal E_t} \sigma (Q) 
\\
& \simeq \rho \sum_{Q\in \mathcal E_t} \lvert Q\rvert 
\lesssim \rho \operatorname e ^{-t} \lvert  Q_0\rvert \simeq \operatorname e ^{-t} \sigma (Q_0)\,,  
\end{align*}
 Our proof is complete. 
\end{proof}

We can now verify \eqref{e.aTL}.  
  We will write 
\begin{align*}
 \sum_{ \substack{Q\in \mathcal L_{a,b} }}  \mathbb E _{Q} w \cdot \mathbf 1_{Q} 
& \lesssim 
 \sum_{S\in \mathcal S(Q_0)}   2 ^{-b} \mathbb E _{S} w   U_S 
\end{align*}
where $ U_S := \sum_{Q\in \mathcal L _{a,b} (S) }   \mathbf 1_{Q}$. 
There is one more variable that is useful for us to hold essentially constant.
Define an event  $  E _{S,0} = \{    U_S  <  \Lambda 2 ^{n+1}\}$, and for $ n >0$, set 
 $ E _{S,n} = \{ 2 ^{n}  \Lambda \le  U_S  < 2 ^{n+1} \Lambda \} $.  Then, we have $ \sigma (E _{S,n}) \lesssim \operatorname e ^{ -c2 ^{n}} \sigma (S)$ for  $ c= c _{\mathcal L}>0$.  Moreover, we can estimate, using a familiar trick,
\begin{align*}
\Bigl[ \sum_{S\in \mathcal S(Q_0)}  \mathbb E _{S} w \cdot   U_S   \Bigr] ^{p'} 
&= 
\Bigl[ \sum_{n=0} ^{\infty } \sum_{S\in \mathcal S(Q_0)}  2 ^{-b - n/p' + n/p'} \mathbb E _{S} w \cdot   U_S   \mathbf 1_{E _{S,n}}  \Bigr] ^{p'} 
\\
& \lesssim \sum_{n=0} ^{\infty }  2 ^{n} 
\Bigl[ \sum_{S\in \mathcal S(Q_0)}  2 ^{-b} \mathbb E _{S} w \cdot   U_S   \mathbf 1_{E _{S,n}}  \Bigr] ^{p'} 
\\
& \lesssim 
 \sum_{n=0} ^{\infty } 2 ^{n} 
 \sum_{S\in \mathcal S(Q_0)}  \Bigl\lvert 2 ^{-b} \mathbb E _{S} w  \cdot  U_S   \mathbf 1_{E _{S,n}}  \Bigr\rvert ^{p'} \,. 
\end{align*}
The first line follows by an appropriate choice of H\"older's inequality, and the second as the sum, for each point $ x \in Q_0$ is a super-geometric series of numbers.  This is wasteful in  $ n$, but 
 by \eqref{e.Ldist2},  
\begin{align*}
 2 ^{n}\int _{E_{S,n}}  \Bigl\lvert  2 ^{-b} \mathbb E _{S} w  U_S    \Bigr\rvert ^{p'} \ \sigma (dx) 
& \lesssim \Lambda ^{p'} 2 ^{-b p' -n} \bigl[ \mathbb E _{S} w \bigr] ^{p'} \sigma (S)  
\\
& \lesssim L ^{p'} 2 ^{-b p' + a (p'-1) -n} w (S) \,. 
\end{align*}
The last line follows by trading out the $ A_p$ characteristic, which is approximately $ 2 ^{a}$ on these stopping cubes. 
We can of course trivially sum this in $ n\ge 0$.  
Now, the $ A_ \infty $ property is decisive.  
We employ the elementary property \eqref{e.14} to estimate 
\begin{align*}
\sum_{S\in \mathcal S(Q_0)} w (S)  &= \int _{Q_0} \sum_{S\in \mathcal S(Q_0)} \frac { w (S) } {\lvert  S\rvert }  \mathbf 1_{S }   \; dx 
\\
& \lesssim \int _{Q_0} M (w \mathbf 1_{Q_0})   \; dx  \le \lVert w\rVert_{A_ \infty } w (Q_0) \,. 
\end{align*}
 Combining estimates, we have  proved \eqref{e.aTL}.

\section{Lerner's Median Inequality; Application to Haar Shift Operators} \label{s.lerner}

We recall definitions for the inequality from \cite{MR2721744}, which applies to a 
measurable function $ \phi $ on $ \mathbb R ^{d}$, and cube $ Q$.  A  median of $ \phi $ restricted to $ Q$, is  a possibly non-unique real number such that 
\begin{equation*}
\max\bigl\{ \lvert \{ x \in Q \;:\;  \phi (x)> m_ \phi  (Q) \}\rvert,\ 
\max\bigl\{ \lvert \{ x \in Q \;:\;  \phi (x)> m_ \phi  (Q) \}\rvert \}\rvert \bigr\} \le \tfrac 12 \lvert  Q\rvert.  
\end{equation*}
For parameter 
$ 0<\lambda <1$, we define an measure of oscillation of $ f$ to be 
\begin{equation}\label{e.w}
\omega _{\lambda } (\phi  ; Q) \coloneqq 
\inf _{c\in \mathbb R }  \bigl( (\phi - c )\mathbf 1_{Q}  \bigr) ^{\ast} (\lambda \lvert  Q\rvert ) \,. 
\end{equation}
Here, $ \phi ^{\ast} $ denotes the non-increasing rearrangement, so that if $ \phi $ is supported on $ Q$, 
$ \phi ^{\ast} (\lambda \lvert  Q\rvert )$ is the $ \lambda ^{\textup{th}}$  percentile of $ \phi $. 
The local sharp maximal function of $ f$ is 
\begin{equation}\label{e.localSharp}
\operatorname M ^{\sharp} _{\lambda; Q } \phi  (x) 
\coloneqq \sup _{Q'\subset Q} \mathbf 1_{Q'} \omega_{\lambda } (\phi , Q')\,. 
\end{equation}

\begin{theorem}[Lerner] \label{t.lerner} Let $ \phi $ be a measurable function on $ \mathbb R ^{d}$, and $ Q_0$ a cube. 
Then, there is a collection of cubes $ \{Q_j^\ell\}$ all dyadic cubes contained in $ Q_0$ so that 
\begin{enumerate}
\item  We have the pointwise inequality 
\begin{equation}\label{e.le}
\lvert  \phi  (x) - m_ \phi  (Q_0)\rvert 
\lesssim 
\operatorname M _{1/4, Q_0} ^{\sharp} \phi  (x) 
+ \sum _{ \ell =1} ^{\infty } \sum _{j}  \omega_{2 ^{-d-2} } (f, \widehat Q_j^\ell ) 
\mathbf 1_{Q_j^\ell} (x) 
\end{equation}
where $ \widehat Q  $ is the parent of dyadic cube $ Q$. 

\item The cubes $ Q_j^\ell$ are disjoint in $ j$, with $ k$ fixed. 

\item Setting $ \Omega _{k} = \bigcup _{j} Q_j^\ell$, we have $ \Omega_{k+1} \subset \Omega_k$, 

\item  $ \lvert   Q_j^\ell \cap \Omega_{k+1}\rvert < \tfrac 12 \lvert  Q_j^\ell\rvert  $. 
\end{enumerate}
\end{theorem}

The collection $ \mathcal L = \{ Q ^{\ell } _{j} \;:\; j, \ell \ge 1\}$ is a collection of dyadic intervals of type $ L$, with 
$ \Lambda  _{\mathcal L}  \simeq 1$.  We will say that $ \mathcal M \subset \mathcal L $ \emph{has generations separated by $ t$} if it is a 
subset of 
\begin{equation} \label{e.generations}
\{Q ^{\ell } _{j} \;:\; j\ge 1\,, \ell \equiv t' \mod t\}, \qquad  0\le t' < t \,. 
\end{equation}

We turn to the application of this inequality to Haar shift operators $ \mathbb S _{\natural} $.  
We will show that  
we can dominate $ \mathbb S_{\natural}$ by a sum of  dyadic positive operators of type $  L$. The sum has a number of terms in 
in controlled by complexity.   Then our technical Theorem~\ref{t.haarShift} 
follows from Proposition~\ref{t.forLerner}.   In order to control the measure of oscillations above, as it now standard 
in the subject, we appeal to a weak-$L^1$ estimate.

\begin{lemma}\label{l.Tweak11}
 Let $ \mathbb S $ be an $ L ^2 $ bounded Haar shift operator of  complexity $ \kappa $.  We then have 
$
\lVert \mathbb S _{\natural} f \rVert_{1, \infty } \lesssim \kappa \lVert f\rVert_{1}
$. 
\end{lemma}

Aside from the linear bound in complexity, this is a standard argument;  details can be found in \cite{1007.4330}*{Proposition 5}.  
We need to make these comments on the estimation of the oscillation terms above for Haar shift operators.  
Recall that $  \mathbb S $ is a Haar shift operator of complexity $ \kappa $. Fixing cube $ Q$, and letting $ Q ^{(\kappa )}$ 
be the $ \kappa $-fold parent of $ Q$, it follows that if  measure $ g$ is \emph{not supported on $ Q ^{(\kappa )}$}, 
that the function $ \mathbb S _{\natural} (g)$ is constant on $Q$.   (Note that this is certainly not true for continuous 
Calder\'on-Zygmund operators,  so this proof seems to be limited to the dyadic setting.)   
Constants do not contribute to the measure of oscillation that we are concerned with, therefore, in estimating 
$  \omega_{2 ^{-d-2} } ( \mathbb S_{\natural} (f \sigma ),  Q  ) $, we can assume that $ f $ is supported on 
$ Q  ^{(\kappa)} $.    Moreover, in seeking to estimate this osciallatory term, we can group all the scales 
inside $ Q$, and appeal to the weak-type estimate.  For the $ \kappa $-scales above $ Q$, we use the size condition \eqref{e.normal}. 
From this, we see that 
\begin{align*}
\Biggl\lvert \Biggl\{x \in  Q   \;:\;  \mathbb S_{\natural} (f \mathbf 1_{ ( Q ) ^{(\kappa)}} \sigma ) 
\ge K \kappa  \mathbb E _{Q }  \lvert  f\rvert \sigma  
+ \sum_{t=1} ^{\kappa +1 } \mathbb E _{(Q) ^{(t)} } \lvert  f\rvert \sigma  
\Biggr\}\Biggr\rvert   
& \le 2 ^{-d-2} \lvert  Q\rvert 
\end{align*}
for a dimensional constant $ K$.  Hence, we conclude that 
\begin{equation*}
 \omega_{2 ^{-d-2} } (\mathbb S_{\natural} (f \sigma ), \widehat Q_j^\ell )
 \lesssim \kappa 
  \mathbb E _{Q } \lvert  f\rvert \sigma  
+ \sum_{t=1} ^{\kappa} \mathbb E _{(Q) ^{(t)} } \lvert  f\rvert \sigma  \,. 
\end{equation*}
From this, it follows that we have the following estimate on the local sharp maximal function, 
$
\operatorname M ^{\sharp} _{\lambda; Q } ( \mathbb S_{\natural} \sigma  f) 
\lesssim \kappa  M  \lvert   f\rvert \sigma 
$.

Now, let $ f$ be supported on a fixed dyadic cube $ Q_1$.  We apply Theorem~\ref{t.lerner} to $ \mathbb S_{\natural} (f \sigma)  $, restricted to 
a cube $ Q_0$, much larger than $ Q_1$. We can estimate 
\begin{align} 
\bigl\lvert \mathbb S _{\natural} (f \sigma )  - 
 \omega_{1/2} (\mathbb S_{\natural} ( f \sigma  ) ,Q_0)
\bigr\rvert 
& \lesssim   \label{e.lernerT<}
\operatorname M ^{\sharp} _{\lambda; Q } ( \mathbb S_{\natural} \sigma  f) 
+  \sum_{\ell=1} ^{\infty } \sum_{j}   \omega_{2 ^{-d-2} } ( \mathbb S_{\natural} (f \sigma ), \widehat Q_j^\ell )   \mathbf 1_{Q ^{\ell} _{j}} 
\end{align}
We have already seen that the local sharp function is bounded by $ \kappa M (f \sigma )$.   The structure of the Haar shift operator 
shows that 
\begin{align*}
 \omega_{2 ^{-d-2} } ( \mathbb S_{\natural} (f \sigma ), \widehat Q_j^\ell ) 
\lesssim  \kappa \mathbb E _{  \widehat Q_j^\ell} \sigma \lvert  f\rvert 
+ \sum_{t=2} ^{\kappa } \mathbb E _{(Q_j^\ell ) ^{(t)}} \sigma \lvert  f\rvert 
\end{align*}
We have again used Lemma~\ref{l.Tweak11}, for the first term on the right. 

We will show that 
\begin{equation} \label{e.LT<}
\Biggl\lVert  
 \sum_{\ell=1} ^{\infty } \sum_{j} \mathbb E _{(Q_j^\ell ) ^{(t)}} \sigma \lvert  f\rvert  \cdot \mathbf 1_{Q ^{\ell } _{j}} 
\Biggr\rVert_{L^p(w)} 
\lesssim C_{p,t}  
 \llbracket w, \sigma \rrbracket _{A_p} 
\bigl\{ 
\|  w \|_{ A _{\infty}}^{1/p'}+\| \sigma \|_{A_ \infty } ^{1/p}\bigr\} \lVert f\rVert_{L ^{p} (\sigma )} \,. 
\end{equation}
This shows that the right hand side of \eqref{e.lernerT<} is bounded with a norm estimate that  depends upon  complexity. 
(It will be exponential.) 
Assuming that $ f$ is compactly supported, and  taking  $ Q_0$ arbitrarily large, we can 
make $ m_ \phi  (Q_0)$ as small as we wish. So by Fatou Theorem, we will have finished the proof.

The maximal function obeys our estimate, see Theorem~\ref{t.max}, bringing our focus to the remaining terms in \eqref{e.lernerT<}.   
The main point is this: The  term in \eqref{e.LT<} is dominated by an 
operators of type $L$, with constant $ \Lambda _{L} \lesssim  2 ^{td}$.  From this, and \eqref{e.typeLstrong}, the required estimate \eqref{e.LT<} follows immediately.

Fix $1\le t \le \kappa +1$, and note the following.  Fix a cube $ R$, and consider $ \mathcal R$, the collection of  those 
$ Q ^{\ell} _{j} \in \mathcal L$ such that $ (Q_j^\ell) ^{(t)} = R$.  The collection $ \mathcal R$ consists of disjoint cubes. 
This means that below, we can work with a set of indices $ (\ell,j) \in \mathbb K $ with the defining property of $ \mathbb K $ being that 
for all pairs of integers $ Q ^{\ell } _{j} \in \mathcal M $ there is a unique $ ( \ell ',j') \in \mathbb K $ with  
$  (Q_j^\ell) ^{(t)}=  (Q _{j'} ^{\ell '}) ^{(t)}$.
We argue that this operator is  of type $ L$ with constant $ \Lambda _{L} \lesssim 2 ^{td}$.  
\begin{equation*}
\sum_{ (\ell,j) \in \mathbb K } \mathbb E _{  ( Q_j^\ell) ^{(t)} } \lvert  f\rvert \sigma    \cdot  \mathbf 1_{(Q ^{ \ell } _{j}) ^{(t)}}  
\end{equation*}

Indeed, for any cube $ R$, we have 
\begin{align*}
\Biggl\lVert  \sum_{\substack{ (\ell , j) \in \mathbb K \\ 
 (Q ^{\ell } _{j}) ^{(t)} \subset R
 }}  \mathbf 1_{( Q ^{\ell } _{j}) ^{(t)}} \Biggr\rVert_{1}
& \le  2 ^{td}
\Biggl\lVert  \sum_{\substack{ (\ell , j) \in \mathbb K \\ 
 Q ^{\ell } _{j}\subset R
 }}  \mathbf 1_{ Q ^{\ell } _{j}} \Biggr\rVert_{1}
 \lesssim 2 ^{td} \lvert  R\rvert 
 \end{align*}
by property (4) of Theorem~\ref{t.lerner}.  This estimate is uniform in $ R$, hence, by a well-known John-Nirenberg arugment, it shows that $ \Lambda _{L} \lesssim 2 ^{td}$, 
completing our proof.

\begin{remark}\label{r.other} 
Our approach gives a proof of the weak-type estimate \eqref{e.zz} from Lerner's inequality, one of the few inequalities 
missing from the papers \cite{1001.4254,MR2628851}.
One should note that the weak-type inequality 
for the dyadic positive operators is quite easy. See \cite{0911.3437,MR719674}.  
\end{remark}

\section{Concluding Remarks} 

Our motivation for writing this  paper is to present positive evidence for this conjecture. 

\begin{conjecture}\label{j.ApAzI} 
For $  T $ an $ L ^{2} (\mathbb R ^{d})$ bounded  Calder\'on-Zygmund Operator and $1<p<\infty$,  
and $ w \in A_p$ it holds that 
\begin{align}
\lVert T _{\natural } f  \rVert _{ L ^{p} (w)} &\le C_{T,p} \|  w \|_{ A _{p}}^{1/p} 
\max \bigl\{ 
\|  w \|_{ A _{\infty}}^{1/p'}, \lVert w ^{- p'+1} \rVert_{A _{\infty } }  ^{1/p}  \bigr\} 
\lVert  f \rVert _{L ^{p} (w)}\,.
\end{align} 
\end{conjecture}

Currently, this is known for the un-truncated operator and  $ p=2$, \cite{1103.5562}. 
For the definition of a Calder\'on-Zygmund operator, we refer the reader to  \emph{op. cit.} 
An attractive part of this conjecture, pointed out to the author by Tuomas Hyt\"onen, is that for canonical examples 
of $ T$, it is a standard part of the subject to have a lower bound on $ \lVert T\rVert_{ L ^{p} (w) \mapsto L ^{p } (w)}$ 
of $ \lVert w\rVert_{A_p} ^{1/p}$.  Thus, the form of the estimate above  quantifies the $ A _{\infty }$ contribution to the norm.  

The corresponding weak-type result is contained in \cite{1103.5229}*{Theorem 12.3}, recalled here as it 
 seems to be the strongest known estimates, in the case of $ p\neq 2$. 

\begin{theorem}\label{t.CZopAinfty}
For $  T $ an $ L ^{2} (\mathbb R ^{d})$ bounded  Calder\'on-Zygmund Operator and $1<p<\infty$, 
\begin{align} \label{e.zz}
\lVert T_{\natural } f  \rVert _{ L ^{p,\infty } (w)} &\le C_{T,p} \|  w \|_{ A _{p}}^{1/p} \|  w \|_{ A _{\infty}}^{1/p'} \lVert  f \rVert _{L ^{p} (w)},\\   \label{e.ZZ}
\lVert T _{\natural} f  \rVert _{ L ^{p } (w)} &\le C_{T,p}\big( \|  w \|_{ A _{p}}^{1/p} \|  w \|_{ A _{\infty}}^{1/p'}+\|w\|_{A_p}^{1/(p-1)}\big) \lVert  f \rVert _{L ^{p} (w)}.
\end{align}
\end{theorem}

It is a curious remark that the extrapolation estimates in \cite{1103.5562} do not give better than the estimate above, 
despite extrapolating from the sharp $ L ^2 $ estimate.   Likewise, our Main Theorem cannot be proved by the 
same elementary arguments from \cite{MR2628851}.

What is required to give a proof of the conjecture? In the singular integral case, we do not have a principle matching that of 
Sawyer's observation in the positive operator case that the strong type norm is dominated by the maximum of two weak type norms. 
Instead, the best results in the two weight case are contained in \cite{0807.0246,0911.3920,1103.5229}. 
The cleanest and simplest argument in 
\cite{1103.5229}*{Theorem 4.7}, which proves a general two-weight estimate for Haar shift operators that is 
linear in complexity.   It is  is largely satisfying, except for the presence of the `non-standard' testing 
condition (4.8).  Indeed, the contribution of the non-standard testing condition to the estimate in \eqref{e.ZZ} 
is the   term $ \|w\|_{A_p}^{1/(p-1)}$ by which we miss the conjecture.  
Any essential strengthening of this Theorem would be interesting, and a potential step towards 
proving the conjecture above.


\end{document}